\documentclass{article}
\usepackage{amssymb,amsmath,amsthm,amscd,verbatim}
\usepackage{hyperref}

\newtheorem{thm}{Theorem}[section]

\newtheorem{lem}[thm]{Lemma}
\newtheorem{df}[thm]{Definition}
\newtheorem{rem}[thm]{Remark}

\newcommand{\la}{\langle}
\newcommand{\ra}{\rangle}

\newcommand{\Iff}{\Leftrightarrow}

\newcommand{\mc}{\mathcal}

\newcommand{\mf}{\mathfrak}

\newcommand{\inter}{\cap}

\renewcommand{\to}{\rightarrow}

\newcommand{\restrict}{\upharpoonright}

\renewcommand{\and}{\,\&\,}

\newcommand{\eps}{\varepsilon}
\newcommand{\nil}{\varnothing}

\begin{document}
\title{Infinite Subsets of Random Sets of Integers}
\author{Bj\o rn Kjos-Hanssen}

\maketitle
\begin{abstract}
There is an infinite subset of a Martin-L\"of random set of integers that does not compute any Martin-L\"of random set of integers. To prove this, we show that each real of positive effective Hausdorff dimension computes an infinite subset of a Martin-L\"of random set of integers, and apply a result of Miller.
\end{abstract}

\section{Introduction}

In Reverse Mathematics \cite{Simpson}, the Stable Ramsey's Theorem for Pairs (SRT$^2_2$) asserts for any $\Delta^0_2$ definition of a set of integers $A$ (the existence of which may not be provable in SRT$^2_2$) the existence of an infinite subset of $A$ or of the complement of $A$. An open problem is whether SRT$^2_2$ implies Weak K\"onig's Lemma (WKL) or Weak Weak K\"onig's Lemma (WWKL). WWKL asserts the existence of Martin-L\"of random sets of integers. These are sets that satisfy all ``computable'' probability laws for the fair-coin distribution (under which each integer has probability $1/2$ of belonging to the set, independently of any other integer). One way to show SRT$^2_2$ implies WWKL would be to show that  

\begin{quote}
$(*)$ there is some $\Delta^0_2$ Martin-L\"of random set of integers $R$ such that each infinite subset of $R$ or its complement computes a Martin-L\"of random set,
\end{quote} 
and show that a proof of $(*)$ goes through in the base system for Reverse Mathematics, Recursive Comprehension (RCA$_0$). The simplest reason why $(*)$ would be true would be
\begin{quote}
$(**)$ each infinite subset of any Martin-L\"of random set of integers computes a Martin-L\"of random set.
\end{quote} 

The statement $(**)$ seemed fairly plausible for a while. Each Martin-L\"of random set of integers $R$ is effectively immune, hence so is each infinite subset of $R$. This observation was used by Hirschfeldt et al. \cite{HJKLS} to show that SRT$^2_2$ implies the principle Diagonally Non-Recursive Functions (DNR), which asserts the existence of diagonally non-recursive functions. It is clear that WWKL should imply DNR; but the fact that WWKL is strictly stronger than DNR was only shown with considerable effort by Ambos-Spies et al. \cite{AKLS}.

However, the assertion $(**)$ is false, the argument having two steps. In the present paper, we show that each real of positive effective Hausdorff dimension computes an infinite subset of a Martin-L\"of random set, and Miller \cite{M} shows that there is a real of effective Hausdorff dimension 1/2 that computes no Martin-L\"of random set. The truth value of $(*)$ remains unknown.

\begin{rem}
A beam splitter is a frequently used component for random number generators. Two photon detectors labeled $0$ and $1$ are used to detect two possible outcomes corresponding to one of two possible paths a photon can take. Thus each photon entering the beam splitter generates one random bit, $0$ or $1$ depending on where the photon is detected. However, photon detectors generally have an absorption efficiency that is less than $100\%$. Can they still be used to produce a random binary sequence? Our main result can be interpreted as the anwer ``no'', if we merely assume that  in an infinite sequence of photons, at least infinitely many of the incoming photons will be detected.  
\end{rem}

\section{Some probability theory}

The proof of our main result uses random closed sets of reals. Their study in computability theory was begun by Barmpalias et al. \cite{BBCDW}, who studied a different distribution from the one we consider here. 

We sometimes consider an integer $K\in\omega=\{0,1,2,\ldots\}$ to be the set $\{0,\ldots,K-1\}$, and write $K^{<\omega}$ and $K^\omega$ for the sets of finite and infinite strings over $K$, respectively. For a tree $T$, $[T]$ denotes the closed set defined by $T$, the set of all infinite paths through $T$. If $\sigma$ is an initial substring of $\tau$ we write $\sigma\preceq\tau$. The concatenation of $\sigma$ and $\tau$ is denoted $\sigma*\tau$, and the length of $\sigma$ is $|\sigma|$.
 
Let $\mc P$ denote the power set operation. \emph{Unless otherwise stated below we have $K=2^k$ for some integer $k\ge 1$.} $K$ plays the role of an alphabet, and a tree is a set of strings over $K$ that is closed under prefixes.

For a real number $0\le\ell<\infty$, let $\lambda_{k,\ell}$ be the distribution with sample space $\mc P(K^{<\omega})$ such that each string in $K^{<\omega}$ has probability $2^{-\ell}$ of belonging to the random set, independently of any other string. We postulate no relationship between $k$ and $\ell$, but note that Theorem \ref{members} below is non-vacuous only for $\ell<k$; a nice case to keep in mind is $\ell=1$, $k=2$.

\begin{lem}\label{new}
For all strings $\rho$, $\sigma$, $\tau$ in $K^{<\omega}$, if $\rho$ is the longest common prefix of $\sigma$ and $\tau$ with $|\sigma|=|\tau|=n$ and $|\rho|=m$, then   
$$\lambda_{k,\ell}\{S:\sigma\in S\text{ and }\tau\in S\}= 2^{\ell(m-2n)}.$$
\end{lem}
\begin{proof}
We have 
$$\lambda_{k,\ell}\{\sigma\in S\text{ and }\tau\in S\}=\lambda_{k,\ell}\{\rho\in S\}\cdot \lambda_{k,\ell}\{\sigma\in S\text{ and }\tau\in S \,|\, \rho\in S\}$$
$$=2^{-\ell m} 2^{-2\ell(n-m)} =2^{\ell(m-2n)}.$$
\end{proof}

The idea of the following Definition \ref{iota} is to think of a string in $K^{<\omega}$ where $K=2^k$ as a string in $2^{<\omega}$ of length a multiple of $k$.

\begin{df}\label{iota}
Let $\iota:K\to 2^{<\omega}$ be defined by the condition that for any $a_i\in\{0,1\}$, $0\le i\le k-1$, 
$$\iota\left(\sum_{i=0}^{k-1} a_i 2^i\right)=\la a_0,\ldots,a_{k-1}\ra.$$ 
For example, if $k=2$ then $\iota(3)=\la 1,1\ra$ and $\iota(2)=\la 1,0\ra$. \\ This is extended to a map on strings,
$\iota: K^{<\omega}\to 2^{<\omega}$, 
by concatenation:
$$\iota(\sigma)=\iota(\sigma(0))*\cdots*\iota(\sigma(|\sigma|-1)).$$ 
For example, if $k=2$ then $\iota(\la 3,2\ra)=\la 1,1,1,0\ra$. \\ Finally $\iota$ is extended to a map on sets of strings, 
$\iota:\mc P(K^{<\omega})\to \mc P(2^{<\omega})$, 
by $$\iota(S)=\{\iota(\sigma):\sigma\in S\}.$$
\end{df}

\begin{df}[$(k,\ell)$-induced distribution]\label{notawk}
For $S\subseteq K^{<\omega}$, $\Gamma_S$, the tree determined by $S$, is the (possibly empty) set of infinite paths through the $\iota$-image of the part of $S$ that is downward closed under prefixes: $$\Gamma_S=\left[\iota\left(\{\sigma\in K^{<\omega}: \forall\tau\preceq\sigma\,\,\tau\in S\}\right)\right].$$

The $(k,\ell)$-induced distribution $\mathbb P_{k,\ell}$ on the set of all closed subsets of $2^\omega$ is defined by
$$\mathbb P_{k,\ell}(E)= \lambda_{k,\ell}\{ S: \Gamma_{S}\in E\},$$
Thus, the probability of a property $E$ of a closed subset of $2^\omega$ is the probability according to $\lambda_{k,\ell}$ that the $\iota$-image of a random subset of $K^{<\omega}$ determines a tree whose set of infinite paths has property $E$.
\end{df}

\begin{lem}[Chebychev-Cantelli, a special case of the Paley-Zygmund Inequality]\label{PZ}
Suppose $X$ is a nonnegative random variable, that is $X:S\to [0,\infty]$ for some sample space $S$, with probability distribution $\mathbb P$ on some $\sigma$-algebra $\mc S\subseteq 2^S$ containing the event $\{X>0\}$. Suppose $\mathbb E[X^2]<\infty$, where $\mathbb E$ denotes expected value.

Then
$$\mathbb P\{X>0\}\ge  \frac{\mathbb E[X]^2}{\mathbb E[X^2]}.$$
\end{lem}
\begin{proof}
Since $X\ge 0$, we have $\mathbb E[X]= \mathbb E[X\cdot \mathbf 1_{\{X>0\}}]$ where $\mathbf 1_{E}$ is the characteristic (indicator) function of $E$ and as is customary we abbreviate $\{X:E(X)\}=\{E(X)\}=E$.
Squaring both sides and applying Cauchy-Schwarz yields
$$\mathbb E[X]^2 = \mathbb E[X\cdot\mathbf 1_{\{X>0\}}]^2\le \mathbb E[X^2]\cdot \mathbb E[(\mathbf 1_{\{X>0\}})^2]=\mathbb E[X^2]\cdot\mathbb P\{X>0\}.$$
\end{proof}

Let the ultrametric $\upsilon$ on $2^\omega$ be defined by $\upsilon(x,y)=2^{-\min\{n:x(n)\ne y(n)\}}$. For a Borel probability measure $\mu$ on $2^\omega$, we write $\mu(\sigma)$ for $\mu([\sigma])$.

\begin{lem}\label{Oct29-2008}
If $\beta>\gamma$ and $\mu$ is a Borel probability measure on $2^\omega$ such that for some constant $c_R$, $\mu(\sigma)\le c_R 2^{-|\sigma|\beta}$ for all binary strings $\sigma$, then $$\iint\frac{d\mu(b)d\mu(a)}{\upsilon(a,b)^\gamma}<\infty.$$
\end{lem}
\begin{proof}
Note that $\mu$ is necessarily non-atomic. Hence we have
$$\phi_\gamma(a):=\int\frac{d\mu(b)}{\upsilon(a,b)^\gamma}\le\sum_{n=0}^\infty 2^{n\gamma} \mu\{b:\upsilon(a,b)=2^{-n}\}=\sum_{n=0}^\infty 2^{n\gamma} \mu[(a\restrict n) * (1-a(n))]$$ 
$$\le c_R \sum_{n=0}^\infty 2^{n\gamma} 2^{-(n+1)\beta}=c_R 2^{-\beta} \sum_{n=0}^\infty 2^{n(\gamma-\beta)}= c_R \frac{2^{-\beta}}{1-2^{\gamma-\beta}}=\frac{c_R}{2^\beta-2^\gamma}.$$
Thus
$$\iint\frac{d\mu(b)d\mu(a)}{\upsilon(a,b)^\gamma}= \int \phi_\gamma(a)d\mu(a)\le  \int \frac{ c_R d\mu(a)}{2^\beta-2^\gamma} =  \frac{ c_R \mu(2^\omega)}{2^\beta-2^\gamma} = \frac{ c_R}{2^\beta-2^\gamma}<\infty.$$

\end{proof}

\begin{lem}\label{oslmuc}
Suppose we are given $\gamma=\frac{\ell}{k}$. Let $\mu$ be any Borel probability measure on $2^\omega$ such that 
$$\iint\frac{d\mu(b)d\mu(a)}{\upsilon(a,b)^\gamma}=c<\infty.$$

Let $\Gamma$ be the (closed set valued) random variable whose value is the outcome of an experiment obeying distribution $\mathbb P_{k,\ell}$, the $(k,\ell)$-induced distribution. For any closed set $A$, we have
$$\mathbb P_{k,\ell}\{\Gamma:\Gamma\inter A\ne\nil\}\ge\frac{\mu(A)^2}{c}.$$
\end{lem}
\begin{proof}
Let $\mf C_n=\{0,1\}^n$, the set of all binary strings of length $n$. Let $S$ denote the random variable that is the outcome of the experiment according to the distribution $\lambda_{k,\ell}$; so $S$ takes values in $\mc P(K^{<\omega})$. Let $n$ be a positive integer that is a multiple of $k$, let $\mf S_n=\{\sigma: |\sigma|=n\text{ and }\iota^{-1}(\sigma)\in S\}$, and $$Y_n=\sum_{\sigma\in\mf S_n, [\sigma]\inter A\ne\nil} \frac{\mu(\sigma)}{2^{-|\sigma|\gamma}}=\sum_{\sigma\in\mf C_n,   [\sigma]\inter A\ne\nil } \mu(\sigma) 2^{n\gamma} \mathbf 1_{\{\sigma\in\mf S_n\}}.$$

Let us write $\mathbb P$ for $\lambda_{k,\ell}$. Note that $\mathbb E[\mathbf 1_{\{\sigma\in\mf S_n\}}]=\mathbb P\{\sigma\in\mf S_n\}=(2^{-\ell})^{n/k} = 2^{-n\gamma}$.

Thus by linearity of expectation $\mathbb E$,
$$\mathbb E[Y_n]=\sum_{\sigma\in\mf C_n,  [\sigma]\inter A\ne\nil}\mu(\sigma)2^{n\gamma}\mathbb E[\mathbf 1_{\{\sigma\in\mf S_n\}}]=\sum_{\sigma\in\mf C_n,  [\sigma]\inter A\ne\nil}\mu(\sigma)\ge\mu(A).$$

Next, $$\mathbb E[Y^2_n]=\sum_{\sigma\in\mf C_n,  [\sigma]\inter A\ne\nil}\sum_{\tau\in\mf C_n,  [\tau]\inter A\ne\nil}\mu(\sigma)\mu(\tau) \,2^{2n\gamma}\, \mathbb P\{\sigma\in\mf S_n\text{ and }\tau\in\mf S_n\},$$
because $\mathbb E[\mathbf 1_{\{\sigma\in\mf S_n\}}\mathbf 1_{\{\tau\in\mf S_n\}}]=\mathbb P\{\sigma\in\mf S_n\text{ and }\tau\in\mf S_n\}$.

Let $\sigma$ and $\tau$ be binary strings and let $m_{\sigma,\tau}$ be the length of the common prefix of $\sigma$ and $\tau$, and let $m'_{\sigma,\tau}=\max\{kq: kq\le m_{\sigma,\tau}, q\in\omega\}$, which is the length of $\iota$ of the common prefix of $\iota^{-1}(\sigma)$ and $\iota^{-1}(\tau)$. Let $\widehat n ={n}/{k}$, and $\widehat m_{\sigma,\tau}=m'_{\sigma,\tau}/k$, which is the length of the common prefix of $\iota^{-1}(\sigma)$ and $\iota^{-1}(\tau)$. 

Note that $\iota^{-1}(\sigma)$ is a $K$-ary string of length $n/k$. By Lemma \ref{new}, $$\mathbb P\{\sigma\in\mf S_n\text{ and }\tau\in\mf S_n\} =\lambda_{k,\ell}(\{S: \iota^{-1}(\sigma)\in S\text{ and }\iota^{-1}(\tau)\in S\}) $$ 
$$\le 2^{\ell(\widehat m_{\sigma,\tau} - 2\widehat n)}= 2^{\gamma (m'_{\sigma,\tau}-2n)}\le 2^{\gamma (m_{\sigma,\tau}-2n)}.$$

Note that if $x\in [\sigma]$ and $y\in [\tau]$ then $\upsilon(x,y)\le 2^{-m_{\sigma,\tau}}$. Therefore
$$\mathbb E[Y^2_n]\le  \sum_{\sigma\in\mf C_n,  [\sigma]\inter A\ne\nil}\sum_{\tau\in\mf C_n,  [\sigma]\inter A\ne\nil}\mu(\sigma)\mu(\tau)\,2^{\gamma m_{\sigma,\tau}}$$
$$\le  \sum_{\sigma\in\mf C_n}\sum_{\tau\in\mf C_n}\mu(\sigma)\mu(\tau)\,2^{\gamma m_{\sigma,\tau}}\le  \iint \frac{d\mu(x)\,d\mu(y)}{\upsilon(x,y)^\gamma}=c.$$

By Lemma \ref{PZ},
$$\mathbb P\{Y_n>0\}\ge\frac{\mathbb E[Y_n]^2}{\mathbb E[Y^2_n]}\ge \frac{\mu(A)^2}{c}.$$

Since $A$ is closed, and since $Y_{n+k}>0$ implies $Y_{n}>0$, we can conclude (letting $n$ still range over multiples of $k$)

$$\mathbb P_{k,\ell}\{A\inter\Gamma\ne\nil\}\ge\mathbb P\{Y_n>0\, \text{ for all } n\} = \lim_{n\to\infty}\mathbb P\{Y_n>0\}\ge \frac{\mu(A)^2}{c}.$$
\end{proof}

A result similar to Lemma \ref{oslmuc}, but for percolation limit sets rather than for the $(k,\ell)$-induced distribution, was obtained by Lyons \cite{Lyons} building on work of Hawkes \cite{Hawkes}.

\section{Martin-L\"of random sets}\label{secBack}

For a real number $0\le\gamma\le 1$, the $\gamma$-weight $\mathrm{wt}_\gamma(C)$ of a set of strings $C$ is defined by
$$\mathrm{wt}_\gamma(C)=\sum_{w\in C} 2^{-|w|\gamma}.$$
A \emph{Martin-L\"of $\gamma$-test} is a uniformly computably enumerable (c.e.) sequence $(U_n)_{n<\omega}$ of sets of strings such that 
$$(\forall n)( \mathrm{wt}_\gamma(U_n)\le 2^{-n}).$$ 

For a set of strings $V$, let $[V]^\preceq=\bigcup\{[\sigma]:\sigma\in V\}$ be the open subset of $2^\omega$ defined by $V$. A real is $\gamma$-random if it does not belong to $\inter_n [U_n]^\preceq$ for any $\gamma$-test $(U_n)_{n<\omega}$. If $\gamma=1$ we simply say that the real, or the set of integers $\{n:x(n)=1\}$, is \emph{Martin-L\"of random}. 

For a (Borel) probability measure $\mu$ and a real $x$, we say that $x$ is \emph{$\mu$-random} if for each sequence $(U_n)_{n<\omega}$ that is uniformly c.e. in $\mu$ and where $\mu [U_n]^\preceq\le 2^{-n}$ for all $n$, we have $x\not\in \inter_n [U_n]^\preceq$. (Note that $\mu$ can be considered as an oracle via an encoding of the reals $\mu([\sigma])$, $\sigma\in 2^{<\omega}$; this definition is due to Reimann and Slaman.)

We say $x$ is \emph{$\gamma$-capacitable} if $x$ is $\mu$-random with respect to some probability measure $\mu$ such that for some $c$,
$$\forall\sigma\,[\mu(\sigma)\le c 2^{-\gamma |\sigma|}].$$

$x$ is \emph{$\gamma$-energy random} if $x$ is $\mu$-random with respect to some probability measure $\mu$ such that 
$$\iint\frac{d\mu(b)d\mu(a)}{\upsilon(a,b)^\gamma}<\infty.$$
If we only require that $x\not\in\bigcap_n [U_n]^\preceq$ when $(U_n)_{n<\omega}$ is c.e., as opposed to c.e. relative to $\mu$, then we say that $x$ is \emph{Hippocrates} $\mu$-random, $\gamma$-capacitable, or $\gamma$-energy random, respectively. As the reader may recall, Hippocrates did not consult the oracle of Delphi, and similarly a test for Hippocrates randomness $U_n$ may not consult the oracle $\mu$.

Effective Hausdorff dimension was introduced by Lutz \cite{Lutz} and is a notion of partial randomness. For example, if the sequence $x_0x_1x_2\cdots$ is Martin-L\"of random, then the sequence $x_00x_10x_20\cdots$ has effective Hausdorff dimension equal to $\frac{1}{2}$. Let $\text{dim}^1_H x$ denote the effective (or constructive) Hausdorff dimension of $x$; then we have $\text{dim}^1_H(x)=\sup\{\gamma: x\text{ is $\gamma$-random}\}$ (Reimann and Stephan \cite{RS01}).

 \begin{thm}[Reimann \cite{Reimann}] \label{Reimann}
 For any real $x\in 2^\omega$,
 $$\dim^1_H x = \sup \{\beta : x\text{ is $\beta$-capacitable}\},$$
 where $\sup \nil = 0$.
 \end{thm}

\begin{lem}\label{ce}
Suppose $U$ is a c.e. set of strings with effective enumeration $U=\bigcup_{s<\omega} U_s$. Then 
$$\{y\in 2^\omega: \{S: y\in\Gamma_{S}\} \subseteq [U]^\preceq\}$$ 
is a $\Sigma^0_1$ class.
\end{lem}
\begin{proof}
For a closed set $\Gamma$, let us write $\Gamma\restrict m$ for $\bigcup\{[\sigma]: |\sigma|=m,\,\,[\sigma]\inter\Gamma\ne\nil\}$. As is usual, for an oracle $A$, a Turing reduction $\{e\}$ and a stage $s$, let $\{e\}^A_s(n)$ denote the output if any of the computation on input $n$ by stage $s$, and $\{e\}^A(n)$ the value by the time the computation halts, if ever.
For some $e$, we have
$$\{y\in 2^\omega: \{S: y\in\Gamma_{S}\} \subseteq [U]^\preceq\}=\{y:\forall S(y\in\Gamma_{S}\to S\in [U]^\preceq)\}$$
$$=\{y:\forall S(\forall m(y\restrict m\in\Gamma_{S}\restrict m)\to S\in [U]^\preceq\}$$
$$=\{y:\forall S\exists m,s(y\restrict m\not\in\Gamma_{S}\restrict m\text{ or } S\in [U_{s}]^\preceq)\}$$
$$=\{y:\forall S(\{e\}^{S\oplus y}(0)\downarrow)\}=\{y:\exists s \forall\sigma\in 2^s(\{e\}^{\sigma\oplus y\restrict s}_s(0)\downarrow)\}.$$
\end{proof}

\begin{df} \label{Oct29Eve}
Let $f_k:\omega\to K^{<\omega}$ be an effective bijection. A closed set is called Martin-L\"of random according to the $(k,\ell)$-induced distribution if it is of the form $\Gamma_{S}$ for some set $S$ such that $f_k^{-1}(S)$ is a Martin-L\"of random subset of $\omega$ with respect to Bernoulli measure with parameter $2^{-\ell}$. 
\end{df}

For $\ell=1$, Definition \ref{Oct29Eve} states that $f_k^{-1}(S)$ is a Martin-L\"of random subset of $\omega$ as defined above. Although we will not need it, one can show that Definition \ref{Oct29Eve} is independent of the choice of $f_k$.

\begin{lem}\label{Jan1-08}
Let $\mu$ be any Borel probability measure on $2^\omega$.
Then for each open set $U$, there is a clopen set $A\subseteq U$ with $\mu(A)\ge \mu(U)-\epsilon$.
\end{lem}
\begin{proof}
Write $U=\bigcup_n D_n$ where the $D_n$ are clopen and disjoint. These $D_n$ exist by topological properties of $2^\omega$.
Then $\infty>1\ge \mu U=\mu \bigcup_n D_n = \sum_n \mu(D_n)$.  Choose $n$ such that $\sum_{k>n}\mu(D_k)\le \epsilon$ and let $A=D_0\cup\cdots\cup D_n$. Then $\mu A=\mu(U)- \sum_{k>n}\mu(D_k)\ge \mu(U)-\epsilon$.
\end{proof}

\begin{thm}\label{members}
Let $\gamma=\frac{\ell}{k}$. Each Hippocrates $\gamma$-energy random real belongs to a Martin-L\"of random closed set under the $(k,\ell)$-induced distribution. 
\end{thm}
\begin{proof}
Let $x$ be any real and suppose $x$ belongs to no Martin-L\"of random closed set according to $\mathbb P_{k,\ell}$. That is, $\{S: x\in\Gamma_{S}\}$ contains no set $S$ that is Martin-L\"of random according to the $(k,\ell)$-distribution. Thus if we let $V_n$ be a universal Martin-L\"of test for $\lambda_{k,\ell}$ then $\{S: x\in\Gamma_{S}\}\subseteq \inter_n [V_n]^\preceq$. So $x\in\inter_{n<\omega}[U_n]^\preceq$, where $$[U_n]^\preceq=\{y\in 2^\omega:\{S:y\in\Gamma_{S}\}\subseteq [V_n]^\preceq\}.$$

We have already seen in Lemma \ref{ce} that $[U_n]^\preceq$ is $\Sigma^0_1$, and by the proof $\{[U_n]^\preceq:n<\omega\}$ is even $\Sigma^0_1$ uniformly in $n$. Thus we may choose $U_n$, $n<\omega$ as uniformly c.e. sets of strings. 
To show $x$ is not $\gamma$-energy random, we must show that for any $\mu$ with $\iint\frac{d\mu(a)d\mu(b)}{\upsilon(a,b)^\gamma}<\infty$, $x$ is not $\mu$-random. We do this by showing that some effectively given subsequence of $\{U_n\}_{n<\omega}$ is a $\mu$-Martin-L\"of test. 

By Lemma \ref{Jan1-08}, for each $\epsilon>0$ there is a clopen set $A\subseteq [U_n]^\preceq$ with $\mu A\ge \mu [U_n]^\preceq-\epsilon$.

Since $$\{S\subseteq K^{<\omega}:\Gamma_{S}\inter [U_n]^\preceq\ne\nil\}=\bigcup_{y\in [U_n]^\preceq} \{S: y\in\Gamma_{S}\}\subseteq [V_n]^\preceq,$$
it follows from Lemma \ref{oslmuc} that $$\frac{\mu(A)^2}{c}\le \lambda_{k,\ell}\{S:\Gamma_{S}\inter A\ne\nil\}\le \lambda_{k,\ell} [V_n]^\preceq\le 2^{-n}$$
so $\mu([U_n]^\preceq)\le \mu(A)+\eps\le \sqrt{c2^{-n}}+\eps$. Since $\eps>0$ was arbitrary, $\mu([U_n]^\preceq)\le \sqrt{c2^{-n}}$. After taking an effective subsequence we can replace $\sqrt{c2^{-n}}$ by $2^{-n}$, and so we are done.
\end{proof}

\begin{thm}\label{anybias}
Each member of any Martin-L\"of random closed set under the $(k,\ell)$-induced distribution is truth-table equivalent to an infinite subset of a Martin-L\"of random set under $\lambda_{k,\ell}$. 
\end{thm}
\begin{proof}
Let $x\in \Gamma_{S}$ where $S$ is Martin-L\"of random under $\lambda_k$. Let $Y=\{ \sigma\in S: \iota(\sigma)\subseteq x\}$ where $\sigma\subseteq x$ denotes that $\sigma$ is an initial segment of $x$. Since $x\in \Gamma_{S}$, $Y$ is infinite. Since $\tau\subseteq x\Iff \iota^{-1}(\tau)\in Y$ for $\tau$ of length a multiple of $k$, and since $\sigma\in Y\Iff \iota(\sigma)\subseteq x$, $x$ is truth-table equivalent to $Y$.
\end{proof}
 
 \begin{thm}\label{main}
 Each real of positive effective Hausdorff dimension computes (and in fact is truth-table equivalent to) an infinite subset of a Martin-L\"of random set of integers.
 \end{thm}
 \begin{proof}
Let $x$ be a real of positive effective Hausdorff dimension, let $\ell=1$ and let $k$ be such that $\dim^1_H(x)>\gamma:= \frac{\ell}{k}=\frac{1}{k}$. By Theorem \ref{Reimann}, $x$ is $\beta$-capacitable for some $\beta>\gamma$. By Lemma \ref{Oct29-2008}, $x$ is $\gamma$-energy random and hence Hippocrates $\gamma$-energy random. By Theorem \ref{members}, $x$ is a member of a Martin-L\"of random closed set under the $(k,\ell)$-induced distribution. By Theorem \ref{anybias}, $x$ is Turing equivalent to an infinite subset $Y$ of a Martin-L\"of random set under $\lambda_{k,\ell}$. Now $\lambda_{k,\ell}$ is the uniform distribution on subsets of $K^{<\omega}$ with success probability $2^{-\ell}=\frac{1}{2}$. Thus $f_k^{-1}(Y)$ is an infinite subset of a Martin-L\"of random set of integers, and truth-table equivalent to $Y$. \end{proof}

\begin{thm}\label{Main}
There is an infinite subset of a Martin-L\"of random set of integers that computes no Martin-L\"of random set of integers.
\end{thm}
\begin{proof}
Miller \cite{M} 
shows that there is a real $x$ of effective Hausdorff dimension 1/2 which does not compute any Martin-L\"of random real. By Theorem \ref{main}, $x$ computes an infinite subset of a Martin-L\"of random set.
\end{proof}

We only needed the case $\ell=1$ for Theorem \ref{Main}. If we let $\ell\to\infty$, we see that even a Martin-L\"of random set of integers with respect to an arbitrarily small probability of membership $2^{-\ell}$ may have an infinite subset that computes no Martin-L\"of random set.

\section*{Acknowledgments}

The author was partially supported by NSF grant DMS-0652669, and thanks Soumik Pal for pointing out the existence of the work of Hawkes and Lyons \cite{Hawkes} \cite{Lyons} and its exposition by M\"orters and Peres \cite{MP}.
\bibliographystyle{plain}
\bibliography{MRL-Kjos-Hanssen-postReview-refs}

\end{document}